\documentclass[12pt]{amsart}
\usepackage{amsmath, amssymb}
\newcommand{\eeq}{\end{equation}}

\usepackage[T1]{fontenc}
\usepackage[dvips]{graphics}
\usepackage{epsfig}
\usepackage{color}

\newtheorem{thm}{Theorem}[section]

\newtheorem{defn}[thm]{Definition}
\newtheorem{conj}[thm]{Conjecture}
\newtheorem{prop}[thm]{Proposition}
\newtheorem{lem}[thm]{Lemma}
\newtheorem{cor}[thm]{Corollary}
\title[The Dixmier-Moeglin equivalence for cocommutative Hopf algebras]{The Dixmier-Moeglin equivalence for cocommutative Hopf algebras of finite Gelfand-Kirillov dimension}

\author{Jason P. Bell}
\thanks{The research of the first-named author  was supported by NSERC grant 611456.}
\address{ Department of Mathematics, University of Waterloo, 
Waterloo, ON, N2L  3G1, CANADA}
\email{jpbell@uwaterloo.ca}

\author{Wing Hong Leung}
\thanks{The research of the second-named author was supported by the Mathematics Department of the Chinese University of Hong Kong and the Undergraduate Summer Research Fellowship from the office of Academic Links of the Chinese University of Hong Kong}

\address{Department of Mathematics, The Chinese University of Hong Kong, Shatin, Hong Kong}
\email{s1155016221@cuhk.edu.hk}

\begin{document}

\begin{abstract} 
Let $k$ be an algebraically closed field of characteristic zero and let $H$ be a noetherian cocommutative Hopf algebra over $k$.   We show that if $H$ has polynomially bounded growth then $H$ satisfies the Dixmier-Moeglin equivalence.  That is, for every prime ideal $P$ in ${\rm Spec}(H)$ we have the equivalences $$P~{\rm primitive}\iff P~{\rm rational}\iff P ~{\rm locally~closed~in}~{\rm Spec}(H).$$  We observe that examples due to Lorenz show that this does not hold without the hypothesis that $H$ have polynomially bounded growth.  We conjecture, more generally, that the Dixmier-Moeglin equivalence holds for all finitely generated complex noetherian Hopf algebras of polynomially bounded growth.
\end{abstract}


\maketitle

%
%
%
%
\section{Introduction} 
The Dixmier-Moeglin equivalence is a fundamental result in the representation theory of enveloping algebras of finite-dimensional Lie algebras.   Generally speaking, the problem of finding the irreducible representations of an algebra is a very difficult and sometimes intractable problem.  As part of a program to deal with this issue, Dixmier proposed that one should first find the kernels of irreducible representations---these are called primitive ideals of the algebra and they form a distinguished subset of the prime spectrum.  This would at least provide a rough classification of the equivalence classes of the irreducible representations.  This approach has been very fruitful in the study of enveloping algebras and their representations.  The problem of characterizing the primitive ideals of enveloping algebras of finite-dimensional complex Lie algebras was solved completely by Dixmier \cite{Dixmier} and Moeglin \cite{Moeglin}.  In this case, we have the following equivalence.
\begin{thm} Let $\mathcal{L}$ be a finite-dimensional complex Lie algebra and let $P$ be a prime in ${\rm Spec}(U(\mathcal{L}))$.  Then the following are equivalent:
\begin{enumerate}
\item[$(1)$] $P$ is primitive;
\item[$(2)$] $\{P\}$ is locally closed in ${\rm Spec}(R)$;
\item[$(3)$] $P$ is rational.
\end{enumerate}
\end{thm}
We recall that a set is locally closed in a topological space if it is of the form $X\setminus Y$, for some closed subsets $X$ and $Y$ of the ambient set; in the case that the set we are considering is a singleton $\{x\}$, we will say that $x$ is locally closed rather than writing $\{x\}$ is locally closed.  
A prime $P$ of a noetherian algebra $A$ is \emph{rational} if the centre of the artinian ring of quotients of $A/P$ is an algebraic extension of the base field.  (Some authors insist that the centre in fact be a finite extension of the base field, although the property is most commonly considered when dealing with algebras over an algebraically closed base field, and there is no distinction in this case.)  The centre of the artinian ring of quotients of a prime noetherian algebra $R$ is called the \emph{extended centroid} of $R$, and we denote it by $\mathcal{C}(R)$.  Thus we see that the Dixmier-Moeglin equivalence gives both a topological and a purely algebraic characterization of the annihilators of simple modules. 

In general, any noetherian algebra with the property that $(1)$--$(3)$ are equivalent for all prime ideals of the algebra is said to satisfy the \emph{Dixmier-Moeglin equivalence}.  To be precise, we give the following definition.  

\begin{defn} {\em Let $R$ be a noetherian algebra.  We say that $R$ satisfies the \emph{Dixmier-Moeglin equivalence} if for every $P\in {\rm Spec}(R)$, the following properties are equivalent:
\begin{enumerate}
\item[($A$)] $P$ is primitive;
\item[($B$)] $P$ is rational;
\item[($C$)] $P$ is locally closed in ${\rm Spec}(R)$.
\end{enumerate}
\label{defn: DM}}
\end{defn}

Since the work of Dixmier and Moeglin, the Dixmier-Moeglin equivalence has been shown to hold for many classes of algebras, including many quantized coordinate rings and quantized enveloping algebras, algebras that satisfy a polynomial identity, and many algebras that come from noncommutative projective geometry \cite{GL, BG, Procesi, BRS}.  The Dixmier-Moeglin equivalence is known not to hold in general, however.  Even in the case of noetherian cocommutative Hopf algebras (of which enveloping algebras of finite-dimensional Lie algebras form an important subclass), it is known that the Dixmier-Moeglin equivalence need not hold.  For example, Lorenz \cite{Lor1} constructed a polycyclic-by-finite group whose complex group algebra is primitive but with the property that $(0)$ is not locally closed in the prime spectrum.  While the Dixmier-Moeglin equivalence need not hold in general it is known that one always has the implications $(C)\implies (A)\implies (B)$ for algebras satisfying the Nullstellensatz; in particular, this holds for countably generated complex noetherian algebras (c.f. \cite[II.7.16]{BG} and \cite[Prop. 6]{Lor2}).

It is interesting to note that the Dixmier-Moeglin equivalence has been shown to hold for large classes of Hopf algebras---aside from the original work of Dixmier and Moeglin on enveloping algebras, it has been shown to hold for certain group algebras \cite{Zal}, for many quantum algebras \cite{GL}, and for a large class of Hopf algebras of Gelfand-Kirillov dimension two \cite{GZ}.  As pointed out earlier, Lorenz \cite{Lor1} has shown that the Dixmier-Moeglin equivalence does not hold in general for Hopf algebras, but we point out that his counter-example has \emph{exponential growth}.  

We quickly recall that for a finitely generated algebra $A$ over a field $k$, if one lets $V\subseteq A$ be a finite-dimensional subspace of $A$ that contains $1$ and generates $A$ as an algebra, then we have the containment $$V\subseteq V^2\subseteq V^3 \subseteq \cdots .$$  We can thus create a \emph{growth function}, $g_V:\mathbb{N}\to \mathbb{N}$ defined by $g_V(n)={\rm dim}_k(V^n)$.  We recall that $A$ has \emph{polynomially bounded growth} if there is some $d>0$ such that $g_V(n)={\rm O}(n^d)$.   It is straightforward to show that the notion of polynomial bounded growth is independent of choice of generating subspace $V$ (c.f. Krause and Lenagan \cite[Lemma 1.1]{KL}) and in the case that $A$ has polynomially bounded growth, the infimum of all $d>0$ for which 
$g_V(n)={\rm O}(n^d)$ is called the \emph{Gelfand-Kirillov dimension of $A$}.  If $A$ does not have polynomially bounded growth the Gelfand-Kirillov dimension is infinite.  If there exists a constant $C>1$ such that $g_V(n)>C^n$ for all $n$ sufficiently large then we say that $A$ has \emph{exponential growth}---again, this definition is independent of our choice of generating subspace $V$.  Finally, an algebra having neither polynomially bounded nor exponential growth is said to have \emph{intermediate growth}.

We note that enveloping algebras of finite-dimensional Lie algebras, group algebras of finitely generated nilpotent groups, and the Hopf algebras considered by Goodearl and Letzter \cite{GL} all have polynomially bounded growth.  In light of this work, we make the following conjecture.
\begin{conj}
Let $H$ be a finitely generated complex noetherian Hopf algebra of finite Gelfand-Kirillov dimension.  Then $H$ satisfies the Dixmier-Moeglin equivalence.
\end{conj}
We are able to prove the conjecture in the special case that the Hopf algebra $H$ is cocommutative.  More specifically, we prove the following result.
\begin{thm}
\label{thm: main}
Let $H$ be a cocommutative noetherian Hopf algebra over an algebraically closed field $k$ of characteristic zero.  If $H$ has finite Gelfand-Kirillov dimension then $H$ satisfies the Dixmier-Moeglin equivalence.
\end{thm}
For the basic background on Hopf algebras, including relevant terminology along with some of the important theory, we refer the reader to the excellent book of Montgomery \cite{Mon}.   Let $H=(H,m,u,\Delta,\epsilon,S)$ be a Hopf algebra with multiplication $m$, comultiplication $\Delta$, unit $\mu$,  counit $\epsilon$, and antipode $S$.  We recall that $H$ is \emph{cocommutative} if $\tau\circ \Delta = \Delta$ where $\tau: H\otimes H\to H\otimes H$ is the linear map that sends $a\otimes b$ to $b\otimes a$ for all $a,b\in H$.  Two of the main examples of cocommutative Hopf algebras are enveloping algebras of Lie algebras and group algebras.  As noted earlier, the Dixmier-Moeglin equivalence was first shown to hold for enveloping algebras of finite-dimensional Lie algebras.  Zalesski\u\i ~ \cite{Zal} showed that the Dixmier-Moeglin equivalence holds for group algebras of finitely generated nilpotent groups.  Both of these classes of algebras have finite Gelfand-Kirillov dimension.  Theorem \ref{thm: main} can thus be seen as a unification of these two results, although we rely heavily on already established results for enveloping algebras in our proof.  The strategy of the proof is to first show that cocommutative noetherian Hopf algebras satisfy the Nullstellensatz.  As remarked earlier, it then only remains to show that ($C$) implies ($A$), where ($A$) and ($C$) are as in Definition \ref{defn: DM}.  We prove this by studying the ideals in the enveloping algebra generated by the primitive elements of $H$ under the action of the grouplike elements of $H$.

\section{Cocommutative Hopf algebras of polynomially bounded growth}
In this section we review the basic results about cocommutative Hopf algebras and we show that the Nullstellensatz holds for such algebras.
 
We recall that if $k$ is a field and $A$ is a $k$-algebra and $H=(H,m,u,\Delta,\epsilon,S)$ is a Hopf $k$-algebra, then a \emph{Hopf action} of $H$ on $A$ is a linear map $\phi: H\otimes A \to A$ that satisfies
$$\phi(h\otimes (ab)) = \sum \phi(h_{(1)}\otimes a)\cdot \phi(h_{(2)}\otimes b),$$ where, in Sweedler notation,
$\Delta(h) = \sum h_{(1)}\otimes h_{(2)}$. We use the more compact notation $h\cdot a$ to denote $\phi(h\otimes a)$.  In the case that $H$ is the group algebra of a group $G$, $\Delta(g)=g\otimes g$ for $g\in G$ and this definition coincides with the ordinary notion of a group action on an algebra.  Given a $k$-algebra with a Hopf $k$-algebra $H$ acting upon it, we can create a $k$-algebra $A\# H$, which is called the \emph{smash product} of $A$ and $H$.  As a $k$-vector space, $A\#H$ is spanned by elements of the form $a\# h$, where $a\in A$ and $h\in H$.  We then have the relations:
\begin{enumerate}
\item $(a+\lambda b)\# h = (a\# h) + \lambda (b\# h)$ for $a,b\in A$, $\lambda\in k$, and $h\in H$;
\item $a\# (h+\lambda h') = (a\# h) + \lambda (a\# h')$ for $a\in A$, $\lambda\in k$, and $h,h'\in H$;
\item $(a\# h)\cdot (b\# h') =\sum  (a \cdot (h_{(1)}\cdot b))\# h_{(2)}h'$ for $a,b\in A$ and $h,h'\in H$.
\end{enumerate}
In the case that $H$ is a group algebra that acts upon $A$, we will write $ag$ instead of writing $a\#g$ for $a\in A$ and $g\in G$.  Also, we will write $a^g$ for $g\cdot a$, and so we have $ga=a^g g$ in $A\# k[G]$.
We recall a well-known result on the structure of cocommutative Hopf algebras over fields of characteristic zero. 
\begin{prop}
\label{prop: structure}
Let $k$ be an algebraically closed field of characteristic zero and let $H$ be a noetherian cocommutative Hopf algebra of finite GK dimension over $k$.  Then $H=U(\mathcal{L})\# k[G]$, where $\mathcal{L}$ is a finite-dimensional Lie algebra over $k$ and $G$ is a finitely generated nilpotent-by-finite group that acts on $\mathcal{L}$ as Lie algebra automorphisms. 
\end{prop}
\begin{proof} By a result of Kostant (see Sweedler \cite[\S 13.1]{Swe}), if $H$ is a pointed cocommutative Hopf algebra over a field of characteristic zero, then $H$ is the smash product of the enveloping algebra of a Lie algebra $\mathcal{L}$ with a group algebra of a group $G$ that acts on $\mathcal{L}$ via Lie algebra automorphisms.  We note that since $k$ is algebraically closed and $H$ is cocommutative, $H$ is pointed.  If, in addition, $H$ has finite GK dimension, then $U(\mathcal{L})$ and $k[G]$ must too, as they are subalgebras of $H$; in particular, $\mathcal{L}$ must be finite-dimensional \cite[Theorem 6.8]{KL} and by Gromov's theorem $G$ must have the property that every finitely generated subgroup is nilpotent-by-finite \cite[Theorem 10.1]{KL}.  If $H$ is noetherian then by a standard faithful flatness argument, we see that $k[G]$ is also noetherian since $H$ is a free right $k[G]$-module.  Since $k[G]$ is noetherian, $G$ is finitely generated, since the poset of subgroups of $G$ embeds in the poset of left ideals of $k[G]$, via the map $K\mapsto \sum_{x\in K} k[G](x-1)$.   Hence $G$ is a finitely generated nilpotent-by-finite group.  The result follows.  
\end{proof}
We recall that a $k$-algebra $A$ is said to satisfy the \emph{Nullstellensatz} if each simple left $A$-module $M$ has the property that the ring of $A$-module endomorphisms of $M$ is algebraic over $k$ and if each prime homomorphic image of $A$ has trivial Jacobson radical.  We next show that a large class of smash products satisfy the Nullstellensatz.  It is conjectured that a group algebra is noetherian if and only if the underlying group is polycyclic-by-finite.  The following result gives evidence that if one takes the smash product of an algebra satisfying the Nullstellensatz with a noetherian group algebra then the resulting algebra should again satisfy the Nullstellensatz.  This result is mostly due to McConnell \cite[Theorem 4.5]{Mc} and is well-known.  
\begin{prop} Let $k$ be an algebraically closed field, let $R$ be a noetherian $k$-algebra that satisfies the Nullstellensatz, and let $G$ be a polycyclic-by-finite group that acts on $R$ as $k$-algebra automorphisms.  Then $R\# k[G]$ is noetherian and satisfies the Nullstellensatz.
\label{prop: Nullextend}
\end{prop}
\begin{proof}
Since $G$ is polycyclic-by-finite we have a chain 
$$\{1\}=G_0\subseteq G_1 \subseteq \cdots \subseteq G_m=G$$ with each $G_i$ normal in $G_{i+1}$ and each $G_{i+1}/G_i$ cyclic.  This chain of subgroups gives rise to a chain of subalgebras
$$R=R_0\subseteq R_1\subseteq \cdots \subseteq R_m=R\# k[G]$$ such that, for each $i$, either $R_{i+1}\cong R_i[x,x^{-1};\sigma]$ or $R_{i+1}$ is a finite free normalizing extension of $R_i$.   

Suppose that either $R\# k[G]$ does not satisfy the Nullstellensatz or that it is not noetherian.  Then there is some smallest $i$ for which at least one of these two properties does not hold for $R_i$.  By assumption, $i>0$.  
If $R_i\cong R_{i-1}[x,x^{-1};\sigma]$ then by \cite[Theorem 4.5]{Mc}, $R_i$ satisfies the Nullstellensatz; it is also noetherian by virtue of being a skew Laurent extension of a noetherian ring.  If, on the other hand, $R_i$ is a finite free normalizing extension of $R_{i-1}$ then $R_i$ is noetherian and has the Jacobson property \cite[9.1.3]{MR}.  A result of Formanek and Jategaonkar \cite[Theorem 4]{FJ} gives that if $M$ is a simple $R_i$-module, then, if we regard $M$ as an $R_{i-1}$-module, it is completely reducible of finite length.  It follows that ${\rm End}_{R_{i-1}}(M)$ is a finite direct product of matrix algebras over $k$ by hypothesis on $R_{i-1}$ and $k$.  Consequently, ${\rm End}_{R_{i-1}}(M)$ is algebraic over $k$, and hence so is ${\rm End}_{R_i}(M)$, being a subalgebra of ${\rm End}_{R_{i-1}}(M)$.  Thus in either case we have that $R_i$ is noetherian and satisfies the Nullstellensatz, a contradiction.  The result follows.
\end{proof}
%
%

We note that the hypothesis that the base field be algebraically closed is not necessary to show that $R\# k[G]$ is noetherian in Proposition \ref{prop: Nullextend}.  We recall that if $R$ is a noetherian algebra that satisfies the Nullstellensatz then for each $P\in {\rm Spec}(R)$ we have the implications
$$(C)\implies (A) \implies (B),$$ where $(A)$, $(B)$, and $(C)$ are as in Definition \ref{defn: DM} (see Brown and Goodearl \cite[II.7.15]{BG}).  This fact and Proposition \ref{prop: Nullextend} show that it is sufficient to consider the implication $(B)\implies (C)$ in obtaining the Dixmier-Moeglin equivalence for cocommutative Hopf algebras of finite Gelfand-Kirillov dimension.  We consider this implication in the next section.

\section{$(B)\implies (C)$}
In this section we show that $(B)$ implies $(C)$ for cocommutative Hopf algebras of finite Gelfand-Kirillov dimension, where $(B)$ and $(C)$ are as in Definition \ref{defn: DM}.  We first recall a useful result of Letzter's, which helps one deal with algebras that are finite free modules over algebras satisfying the Dixmier-Moeglin equivalence.

\begin{thm} (Letzter \cite{Let}) \label{thm: Letzter} Let $R\subseteq S$ be noetherian algebras and suppose that $S$ is a finite free $R$-module on both sides.  Then the following hold: \begin{enumerate}
\item[(i)] $(A)\implies (B)$ for $P\in {\rm Spec}(R)~ \iff ~(A)\implies (B)$ for $P\in {\rm Spec}(S)$;
\item[(ii)] $(B)\implies (A)$ for $P\in {\rm Spec}(R)~\iff ~(B)\implies (A)$ for $P\in {\rm Spec}(S)$;
\item[(iii)] if $S$ has finite GK dimension then:\\
 $(A)\implies (C)$ for $P\in {\rm Spec}(R)~\iff ~(A)\implies (C)$ for $P\in {\rm Spec}(S)$,
\end{enumerate}
where (A), (B), and (C) are as in Definition \ref{defn: DM}.
\end{thm}
Our next result deals with algebraic groups acting rationally on an algebra.  We recall that an affine algebraic group $G$ acts \emph{rationally} on a $k$-algebra $A$ if $A$ can be written as a directed union of finite-dimensional $G$-invariant subspaces with the property that for each such invariant subspace $V$, the natural map from $G$ to ${\rm GL}(V)$ is a morphism of algebraic varieties.  We refer the reader to the book of Brown and Goodearl \cite[II.2.6]{BG} for further details.  The following lemma is a consequence of standard results.

\begin{lem} Let $G$ be an affine algebraic group that acts rationally on a noetherian algebra $A$, let $H$ be a subgroup of $G$ and let $\overline{H}$ denote the Zariski closure of $H$ in $G$.  If $I$ is a two-sided ideal of $A$.  Then $I$ is $H$-invariant if and only if $I$ is $\overline{H}$-invariant.
\label{lem: bar}
\end{lem}
\begin{proof}  
If $I$ is $\overline{H}$-invariant then $I$ is clearly also $H$-invariant.  Conversely, suppose that $I$ is $H$-invariant and let $a_1,\ldots ,a_d$ be generators for $I$.  Then since $G$ acts rationally on $A$, there is a finite-dimensional $G$-invariant vector subspace $W$ of $A$ that contains $a_1,\ldots ,a_d$.   Since $I$ is $H$-invariant and $W$ is $G$-invariant, we have that $I\cap W$ is $H$-invariant.  A standard result (see \cite[Proposition 2.10]{Fogarty}) gives that the stabilizer of $I\cap W$ in $G$ is Zariski closed.  Since this stabilizer contains $H$, it must therefore also contain $\overline{H}$.   Thus $g\cdot a_i\in I$ for every $g\in \overline{H}$ and for every $i\in \{1,\ldots ,d\}$.  It follows that $I$ is 
$\overline{H}$-invariant.  
%
\end{proof}
The next result shows that a dichotomy appears when we look at prime homomorphic images of a smash product of an artinian algebra with the group algebra of a finitely generated nilpotent group.  Namely, the homomorphic image is either simple, or it has a large centre and every nonzero ideal must intersect the centre non-trivially.  This result has some relation to a ``polycentrality'' result of Roseblade and Smith \cite[Theorem 11.3.1]{Passman}.
\begin{prop} Let $k$ be a field, let $B$ be an artinian $k$-algebra, and let $G$ be a finitely generated nilpotent group.  If $A=(B\# k[G])/P$ is a semiprime homomorphic image of $B\# k[G]$ then every nonzero ideal of $A$ intersects the centre of $A$ non-trivially.
\label{prop: semiprime}
\end{prop}
\begin{proof}
We let $$G=G_m \supseteq G_{m-1} \supseteq  \cdots \supseteq G_0 =\{1\}$$ be the upper central series of $G$ and let $A_i$ denote the image of $B\# k[G_i]$ in $A$ for $i=0,\ldots ,m$.  By Proposition \ref{prop: Nullextend}, we see that each $A_i$ is noetherian.  We claim that each $A_i$ is in fact semiprime.  To see this, we note that since $P$ is a semiprime ideal of $B\#k[G]$, $P$ is an intersection of prime ideals of $B\#k[G]$ and hence $P\cap (B\# k[G_i])$ is an intersection of $G$-prime ideals of $B\# k[G_i]$.  Since $G$-prime ideals of noetherian algebras are semiprime, we have that each $A_i$ is a semiprime noetherian ring.

We next claim that if $J$ is a nonzero $G$-stable ideal of $A_i$ for some $i\in \{0,1,\ldots ,m\}$ then $J\cap Z(A)$ is nonzero.  Taking $i=m$ then gives the desired result.  We argue by induction on $i$.  If $i=0$ then $J$ is a nonzero ideal of $A_0$.  Since $A_0$ is semprime and is a homomorphic image of $B$, we see that $A_0$ is a semisimple artinian ring.   Since $J$ is $G$-stable, we see that $J$ is a direct sum of a $G$-stable collection of Wedderburn components of $B$.  Taking the sum of the identity matrices in these components then gives a nonzero element in $J\cap Z(A)$.   We now suppose that the claim holds whenever $i<d$ and consider the case when $i=d$.

We let $I$ be a $G$-stable ideal of $A_d$.  Fix a transversal $\{g_{\lambda}\}_{\lambda \in \Lambda}$ for $G_{d-1}$ in $G_d$ such that $g_{\lambda_0}=1$ for some $\lambda_0\in \Lambda$.  Given $x\in A_d$, we may write 
\begin{equation}
\label{eq: x}
x=\sum_{\lambda \in \Lambda} b_{\lambda}(x) g_{\lambda}
\end{equation} with each $b_{\lambda}(x)\in A_{d-1}$ and all but finitely many $b_{\lambda}(x)$ equal to zero.   This expression is not in general unique, since we are considering elements modulo $P$.  For each $x\in A_d$, we define the \emph{length} of $x$ to be the quantity
\begin{equation}
\label{eq: L}
L(x):=\inf\#\{\lambda \in \Lambda \colon b_{\lambda}(x)\neq 0\},
\end{equation} where the infimum is taken over all expressions for $x$ of the form given in Equation (\ref{eq: x}).
 We note that $L(x)>0$ unless $x=0$.  
 
Let $D=\inf\{L(x)\colon x\in I\setminus \{0\}\}$ and let $J$ denote the collection of all elements of the form $b_{\lambda_0}(x)$ where $x$ ranges over all elements of $I$ that have length at most $D$ and we only take the $b_{\lambda_0}(x)$ appearing in an expression where the infimum in Equation (\ref{eq: L}) is realized.  Then $J$ is a nonzero ideal of $A_{d-1}$.   (By right multiplying an element $x\in I$ by an appropriate element of $G$, we can always obtain an element $y\in I$ of the same length with $b_{\lambda_0}(y)\neq 0$.)  Moreover, we have that $J^g=J$.  To see this, observe that if $b\in J$ is nonzero, then there is some nonzero $x\in I$ that has an expression 
$$x=b+\sum_{j=2}^D b_j g_{\mu_j}$$ for some $\mu_2,\ldots ,\mu_D\in \Lambda\setminus \{\lambda_0\}$ and $b_2,\ldots ,b_D\in A_{d-1}$.   We note that for $g\in G$ and $\lambda\in\Lambda$, we have $$g_{\lambda}^g =[g,g_{\lambda}^{-1}]g_{\lambda}\in A_{d-1}g_{\lambda}.$$ Then 
$x-x^g\in I$ and 
$$x-x^g = (b-b^g) +\sum_{j=2}^D (b_j - b_j^g [g,g_{\mu_j}^{-1}])g_{\mu_j},$$ and so $L(x-x^g)\le D$ and thus $b-b^g\in J$.  It follows that $J=J^g$.
  
By the inductive hypothesis, there is some nonzero $z\in J$ that is central in $A$.  Moreover, by definition of $J$ we have
that there is some nonzero $y\in I$ such that 
$$y=z+\sum_{i=2}^D c_i g_{\nu_i}$$ for some
$\nu_2,\ldots ,\nu_D\in \Lambda\setminus \{\lambda_0\}$ and $c_2,\ldots ,c_D\in A_{d-1}$.
Then for $g\in G$ we have $y-y^g = \sum_{i=2}^D (c_i -c_i^g    [g,g_{\nu_i}^{-1}])g_{\nu_i}$, and so 
$L(y-y^g)<D$ and $y-y^g\in I$.  By minimality of $D$, we see that $y^g=y$ for every $g$ in $G$.  Similarly, if $r\in A_0$ then 
$L(yr-ry)<D$ and so we see that $ry-ry=0$ for all $r\in A_0$.  Since $A$ is generated by $A_0$ and the image of $G$ in $A$, we see that $y$ is central.  The result now follows by induction. 
\end{proof}
\begin{cor} Let $k$ be an algebraically closed field and let $R$ be a prime noetherian $k$-algebra that satisfies the Nullstellensatz and the Dixmier-Moeglin equivalence and let $G$ be a finitely generated nilpotent-by-finite group that acts on $R$.  Suppose that there is a finitely generated $k$-vector space $V$ of $R$ such that $V$ generates $R$ as a $k$-algebra and $V^g=V$ for all $g\in G$. Then if $R$ has finite Gelfand-Kirillov dimension then $R\# k[G]$ satisfies the Dixmier-Moeglin equivalence.
\label{cor: dix}
\end{cor}
\begin{proof}  Since $R$ satisfies the Nullstellensatz, so does $R\# k[G]$ by Proposition \ref{prop: Nullextend}.  Let $N$ be a nilpotent subgroup of $G$ of finite index.  Note that $R\# k[G]$ is a finite free module over $R\# k[N]$.  We also note that $R\# k[G]$ has finite Gelfand-Kirillov dimension.  To see this let $W\subseteq k[G]$ be a finite-dimensional $k$-vector space that contains $1$, is spanned by elements of $G$, and that generates $k[G]$ as a group.  Then $U:=V+W$ generates $R\#k[G]$ as a $k$-algebra.  By assumption, $WV=VW$ and so $U^n \subseteq  (k+V)^n W^n$ and so ${\rm dim}(U^n)\le {\rm dim}((k+V)^n) \cdot {\rm dim}(W^n)$.  Since $R$ has finite Gelfand-Kirillov dimension and since $G$ is nilpotent-by-finite, we see that both ${\rm dim}((k+V)^n)$ and ${\rm dim}(W^n)$ are polynomially bounded functions of $n$ and so $R\#k[G]$ has finite Gelfand-Kirillov dimension.  Thus by Theorem \ref{thm: Letzter}, we see that $R\# k[G]$ satisfies the Dixmier-Moeglin equivalence if and only if $R\# k[N]$ satisfies the Dixmier-Moeglin equivalence.  Thus we may assume without loss of generality that $G$ is nilpotent.  

Since $R\#k[G]$ satisfies the Nullstellensatz, it is sufficient to show that if $P$ is a prime ideal of $R\# k[G]$ and the extended centroid of $(R\# k[G])/P$ is equal to $k$ then the set of height one prime ideals in $R\#k[G]/P$ is finite.
Let $P$ be a prime ideal of $R\# k[G]$ and suppose that the extended centroid of $(R\# k[G])/P$ is equal to $k$.  Let $Q=P\cap R$ and let $\overline{G}$ denote the Zariski closure of the image of $G$ in ${\rm GL}(V)$.  Then $\overline{G}$ is an affine algebraic group that acts rationally on $R$ (see the book of Brown and Goodearl \cite[II.2.6]{BG} for relevant definitions) and by Lemma \ref{lem: bar} $Q$ is a $\overline{G}$-prime of $R$ and consequently $Q$ is a semiprime ideal of $R$.  Since the extended centroid of $(R\# k[G])/P$ is equal to $k$, we see that the $G$-invariants of the extended centroid of $R/Q$ is also equal to $k$ and in particular the $\overline{G}$-invariants of the extended centroid of $R/Q$ is also $k$.  

In particular, $Q$ is a $\overline{G}$-rational prime ideal of $R$.  By a result of Lorenz \cite[Theorem 1]{Lor2} there is a surjection from the set of rational prime ideals of $R$ onto the set of $\overline{G}$-rational ideals of $R$ given by $J\mapsto \gamma(J):=\bigcap_{g\in \overline{G}} J^g$ and hence there is a rational prime $J$ of $R$ such that $\gamma(J)=Q$.  Since $R$ satisfies the Dixmier-Moeglin equivalence, we see that $J$ is locally closed in ${\rm Spec}(R)$ and thus by another result of Lorenz \cite[Theorem 1]{Lor3} we have $Q=\gamma(J)$ is locally closed in the $\overline{G}$-spectrum of $R$.  Let $S=R/Q$.  Then $T:=R\#k[G]/P$ may be regarded as a prime homomorphic image of $S\#k[G]$ with the property that the composition of the natural map from $S$ to $S\#k[G]$ and the natural surjection from $S\#k[G]$ to $T$ is injective.  Henceforth, we identify $S$ with its image in $T$.
It follows if $\mathcal{U}$ is the set of height one prime ideals of $T$ that intersect $S$ non-trivially then the intersection of $\mathcal{U}$ is nonzero as it contains the intersection of the nonzero $\overline{G}$-primes of $S$, which is nonzero since $Q$ is locally closed in the $\overline{G}$-spectrum of $R$.

 

Let $\mathcal{T}$ denote the set of height one primes of $T$ that intersect $S$ trivially.  Then the set of prime ideals in $\mathcal{T}$ survives in the localization of $T$ obtained by inverting the regular elements of $S$, which is a prime homomorphic image of $Q(S)\# k[G]$.  Since $Q(S)$ is semiprime artinian, By Proposition \ref{prop: semiprime}, this localization either has a non-trivial centre or it is simple.  Since $T$ has extended centroid equal to $k$, we see that $\mathcal{T}$ must be empty.  Thus if $T$ has an extended centroid that is algebraic over $k$ then $\mathcal{S}\cup \mathcal{T}$ must be finite and so $P$ is locally closed in ${\rm Spec}(R\# k[G])$.  
\end{proof} 

We may now deduce our main result from Corollary \ref{cor: dix}.
\begin{proof}[Proof of Theorem \ref{thm: main}]
By Proposition \ref{prop: structure}, we have that $H=U(\mathcal{L})\# k[G]$, where $\mathcal{L}$ is a finite-dimensional Lie algebra and $G$ is a nilpotent-by-finite group with $\mathcal{L}^g = \mathcal{L}\subseteq U(\mathcal{L})$ for all $g\in G$.  Since $U(\mathcal{L})$ satisfies the Nullstellensatz \cite[Corollary 9.4.22]{MR}, by Corollary \ref{cor: dix}, $H$ satisfies the Dixmier-Moeglin equivalence.  
%
\end{proof}
\section*{Acknowledgments} We thank Martin Lorenz for many helpful comments and suggestions.  In particular, he pointed out a strengthening of our original Proposition \ref{prop: semiprime} along with many helpful references that considerably shortened and improved some of the proofs.

 \end{document}